\documentclass[a4paper]{article}
\usepackage{amsmath,amssymb,epsfig,arydshln,float,amsfonts,latexsym,color,algorithm,fancyhdr, amsthm}
\usepackage[colorlinks=true,breaklinks=true,linkcolor=lightblue,citecolor=lightblue]{hyperref}

\definecolor{lightblue}{rgb}{0.22,0.45,0.70}

\definecolor{green}{rgb}{0.22,0.55,0.00}

\newtheorem{lemma}{Lemma}[section]
\newtheorem{theorem}{Theorem}[section]

\newtheorem{remark}{Remark}[section]

\numberwithin{equation}{section}

\newcommand\ff{\boldsymbol{f}}

\newcommand\bu{\boldsymbol{u}}
\newcommand\bv{\boldsymbol{v}}

\newcommand\cT{\mathcal{T}}

\newcommand\R{\mathbb{R}}
\newcommand\K{\mathbf{K}}

\newcommand{\norm}[1]{\ensuremath{\left\|#1\right\|}}
\newcommand*\nnorm[1]{|\!|\!| #1 |\!|\!|}

\renewcommand\O{\Omega}
\newcommand\G{\Gamma}

\renewcommand\H{\mathrm{H}}
\renewcommand\L{\mathrm{L}}
\newcommand\W{\mathrm{W}}
\newcommand\Q{\mathrm{Q}}
\newcommand\Z{\mathrm{Z}}

\newcommand\LO{\L^2(\O)}
\newcommand\LOO{\L_{0}^2(\O)}
\newcommand\vdiv{\mathop{\mathrm{div}}\nolimits}
\newcommand\bdiv{\mathop{\mathbf{div}}\nolimits}

\newcommand\HO{\H^1(\O)}
\newcommand\HCUO{\H_{0}^1(\O)}
\newcommand\HsO{\H^s(\O)}
\newcommand\HusO{\H^{1+s}(\O)}

\newcommand\bn{\boldsymbol{n}}

\newcommand\beps{\boldsymbol{\varepsilon}}

\numberwithin{theorem}{section}
\numberwithin{lemma}{section}
\numberwithin{corollary}{section}
\numberwithin{proposition}{section}
\numberwithin{remark}{section}

\newcommand\curl{\mathop{\mathbf{curl}}\nolimits}

\newcommand\CT{{\mathcal T}}

\newcommand\CP{{\mathcal P}}
\newcommand\bI{\mathbf I}

\newcommand\cero{\boldsymbol{0}}

\newcommand\bomega{\boldsymbol{\omega}}
\newcommand\btheta{\boldsymbol{\theta}}
\newcommand\bbP{\mathbb{P}}

\allowdisplaybreaks
\title{Incorporating variable viscosity in vorticity-based\\
formulations for Brinkman equations\thanks{
Funding: CONICYT-Chile through FONDECYT project 11160706,
through Becas-Chile Programme for foreign students
and through the project AFB170001 of the PIA Program:
Concurso Apoyo a Centros Cient\'ificos y
Tecnol\'ogicos de Excelencia con Financiamiento Basal.}}
\author{Ver\'onica Anaya\thanks{GIMNAP, Departamento de Matem\'atica,
Universidad del B\'io-B\'io, Casilla 5-C, Concepci\'on, Chile. E-mail:  
{\tt vanaya@ubiobio.cl}.}, \quad  
 Bryan G\'omez-Vargas\thanks{Secci\'on de Matem\'atica,
 Sede de Occidente, Universidad de Costa Rica, San Ram\'on,
 Costa Rica. Present address:  CI$^2$MA and Departamento de Ingenier\'\i a Matem\' atica,
Universidad de Concepci\' on, Casilla 160-C, Concepci\' on, Chile,
email: {\tt bryan.gomezvargas@ucr.ac.cr}.}, \quad 
David Mora\thanks{GIMNAP, Departamento de Matem\'atica,
Universidad del B\'io-B\'io, Casilla 5-C, Concepci\'on, Chile; and CI$^2$MA,
Universidad de Concepci\' on, Concepci\' on, Chile. E-mail:  
{\tt dmora@ubiobio.cl}.}, \quad 
 Ricardo Ruiz-Baier\thanks{Mathematical Institute,
University of Oxford, A. Wiles Building, 
Woodstock Road, Oxford OX2 6GG, UK. E-mail: {\tt ruizbaier@maths.ox.ac.uk}.}}  
  
\date{}

\begin{document}
\maketitle

\begin{abstract} 
In this brief note, we introduce a non-symmetric mixed finite
element formulation for Brinkman equations written in terms
of velocity, vorticity and pressure with non-constant viscosity.
The analysis is performed by the classical Babu\v ska-Brezzi theory,
and we state that any inf-sup stable finite element pair for Stokes
approximating velocity and pressure can be coupled with a generic
discrete space of arbitrary order for the vorticity. We establish optimal
a priori error estimates which are further confirmed through computational examples. 
%
\end{abstract}



\section{Introduction}
Formulations for flow equations that use vorticity as an additional unknown enjoy many appealing features \cite{speziale87}, and starting from the works \cite{chang90,DSScmame03}, they have been employed in many instances (see e.g. \cite{ACPT07,ACT,anaya15b,amoura07,bernardi06,SSIMA15,cockburn,vassilevski14,anaya16,tsai05}). However, a major limitation in all of these contributions, in comparison with competing formulations using solely the primal variables, is that the transformation of the momentum equation introducing vorticity (and subsequently using a convenient structure of the problem to analyse its mathematical properties and devising suitable numerical schemes) is only valid when the viscosity is constant. Plus, a number of applications including Stokes flow and coupled thermal or thermo-haline effects with Brinkman flows (see e.g. \cite{john15,rudi17,trias13} and \cite{patil82,payne99,woodfield19}, respectively) depend strongly on marked spatial distributions of viscosity.  

In this brief note, we provide a way of incorporating variable viscosities while keeping vorticity as field variable. The resulting non-symmetric formulation is augmented via least-squares terms involving the constitutive equation and mass conservation equation and subsequently the problem maintains a saddle-point structure amenable to analysis through classical tools from mixed methods (under the assumption that the viscosity is regular enough). Even if we have decided to provide all steps for the specific case of Brinkman equations, the same ideas in principle carry over to other vorticity-based models such as Oseen, Navier-Stokes, interfacial flows, and coupled Boussinesq or flow-transport problems. 

The main advantages of the propose scheme are the direct approximation of vorticity without invoking any postprocessing, and also the simplicity of the analysis and implementation. Indeed, one can use standard inf-sup stable finite elements for the Stokes equations plus any conforming discrete space for vorticity. 

\noindent\textbf{Outline.} In Section~\ref{sec:model}, we recall the governing equations and state the least-squares--based augmented formulation. There we also perform the solvability analysis employing standard arguments from the Babu\v ska--Brezzi theory. The finite element discretisation is presented in Section~\ref{sec:FEM}, where we also write a stability analysis and derive optimal error estimates. A few numerical tests illustrating the convergence of the proposed method  are finally reported in Section~\ref{sec:results}.

\section{Variable viscosity Brinkman equations}\label{sec:model}
Let $\O$ be a  bounded domain of $\R^3$ with Lipschitz boundary $\G=\partial\O$, and 
let us write the following version of the Brinkman equations with variable viscosity 
where the unknowns are velocity $\bu$, vorticity $\bomega$, 
and pressure $p$ of the incompressible viscous fluid
\begin{align}\label{eq:momentum}
\nu\K^{-1}\bu+\nu\curl\bomega-2\beps(\bu)\nabla\nu+\nabla p & = 
  \ff & \mbox{ in } \O, \\
  \bomega-\curl\bu&=\cero & \mbox{ in } \O,\label{eq:constitutive} \\ 
  \vdiv\bu & =  0 & \mbox{ in } \O,  \label{eq:mass}\\ 
  \bu & = \cero&    \mbox{ on } \G,\label{eq:bc1}\\
  (p,1)_{0,\O}&=0.&\label{eq:bc2}
 \end{align}
The kinematic viscosity is assumed such that 
$\nu\in \W^{1,\infty}(\Omega)$ and
\begin{equation}\label{nubound}
0<\nu_0\le\nu\le\nu_1.
\end{equation}
Moreover, $\ff\in \L^2(\Omega)^3$ is a force density and 
 $\K\in \L^{\infty}(\Omega)^{3\times3}$
is the (symmetric and uniformly positive definite) tensor of permeability. 
In particular, there exist $\sigma_{\min},\sigma_{\max}>0$ such that
$$\sigma_{\min}\vert \bv\vert^2\le\bv^{t}\K^{-1}\bv\le\sigma_{\max}
\vert\bv\vert^2\qquad\forall\bv\in\R^3.$$ 
Instead of $\nu\K^{-1}$ some works equivalently use $\hat{\K}^{-1}$ 
as the drag coefficient in the momentum equation, where $\hat{\K} = \K/\nu$. 
Note that  \eqref{eq:momentum} can be derived from the usual momentum 
equation by invoking the identity 
$$-2\bdiv(\nu\beps(\bu))=-2\nu\bdiv(\beps(\bu))-2\beps(\bu)\nabla\nu
=-\nu\Delta\bu-2\beps(\bu)\nabla\nu=\nu\curl(\curl\bu)-2\beps(\bu)\nabla\nu,$$
where $\beps(\bu)$  is the strain rate tensor and where we have also used 
\eqref{eq:mass} and the additional identity 
\begin{equation}\label{curl-curl}
\curl(\curl\bv)=-\Delta\bv+\nabla(\vdiv\bv).
\end{equation}

\subsection{Variational formulation and preliminary results}
For any $s\geq 0$, the notation $\norm{\cdot}_{s,\O}$ stands
for the norm of the Hilbertian Sobolev spaces $\HsO$ or
$\HsO^3$, with the usual convention $\H^0(\O):=\LO$.  
We also endow the space $\HCUO^3$ with the following norm:
$$\nnorm{\bv}_{1,\O}^2:=\Vert\bv\Vert_{0,\O}^2
+\Vert\curl\bv\Vert_{0,\O}^2+\Vert\vdiv\bv\Vert_{0,\O}^2.$$
We note that in $\HCUO^3$, the above norm is equivalent
to the usual norm. In particular, we have that there
exists a positive constant $C_{pf}$ such that:
\begin{equation}\label{poinc}
\Vert\bv\Vert_{1,\Omega}^2\le C_{pf}(\Vert\curl\bv\Vert_{0,\Omega}^2
+\Vert\vdiv\bv\Vert_{0,\O}^2)\qquad\forall\bv\in\HCUO^3,
\end{equation}
the above inequality is a consequence of the identity $\Vert\nabla\bv\Vert_{0,\O}^2=\Vert\curl\bv\Vert_{0,\Omega}^2
+\Vert\vdiv\bv\Vert_{0,\O}^2$ which follows from \eqref{curl-curl}
and the Poincar\'e inequality.

\noindent Testing \eqref{eq:momentum}-\eqref{eq:mass} 
appropriately, using Green's formula in the following version (see \cite[Thm.~2.11]{gr-1986})
\begin{equation*}
\int_{\O}\curl\bomega\cdot\bv=\int_{\O}\bomega\cdot\curl\bv+\langle\bomega\times\bn,\bv\rangle_{\partial\O},
\end{equation*}
and applying the boundary conditions \eqref{eq:bc1}-\eqref{eq:bc2}, we get the following weak formulation
\begin{align*}
\int_{\O}\nu\K^{-1}\bu\cdot\bv-2\int_{\O}\beps(\bu)\nabla\nu\cdot\bv
+\int_{\O}\nu\bomega\cdot\curl\bv+\int_{\Omega}\bomega\cdot(\nabla\nu\times\bv)-\int_{\O}p\vdiv\bv&=
\int_{\O}\ff\cdot\bv&
\forall\bv\in\HCUO^3,\\ 
\int_{\O}\nu\btheta\cdot\curl\bu-\int_{\O}\nu\bomega\cdot\btheta&=\,0&\forall\btheta\in\LO^3,\\
-\int_{\O}q\vdiv\bu&=\,0&\forall q\in\LOO,
\end{align*}
where $\LOO:=\{q\in\LO: (q,1)_{0,\Omega}=0\}$. Then, we proceed to augment this 
formulation with the
following residual terms arising from equations \eqref{eq:constitutive}
and \eqref{eq:mass}:
\begin{align}\label{least0}
\kappa_1\nu_0\int_{\O}(\curl\bu-\bomega)\cdot\curl\bv&=0\qquad\forall\bv\in\HCUO^3,\\
\label{least1}
\kappa_2\int_{\O}\vdiv\bu\vdiv\bv&=0\qquad\forall\bv\in\HCUO^3,
\end{align}
with $\nu_0>0$ (cf. \eqref{nubound}), and
where $\kappa_1$ and $\kappa_2$ are positive parameters to be specified later. Then, the 
augmented formulation reads:  
{\em Find $((\bu,\bomega),p)\in(\HCUO^3\times\LO^3)\times\LOO$ such that}
\begin{equation}\label{probform2}
\begin{split}
A((\bu,\bomega),(\bv,\btheta))+B((\bv,\btheta),p)=&\;G(\bv,
\btheta)\qquad\forall(\bv,\btheta)\in\HCUO^3\times\LO^3,\\
B((\bu,\bomega),q)=&\;0\qquad\forall q\in\LOO,
\end{split}
\end{equation}
where the bilinear forms and the linear functional are defined by
\begin{align}
A((\bu,\bomega),(\bv,\btheta))& := \int_{\O}\nu\K^{-1}\bu\cdot\bv
+\int_{\O}\nu\bomega\cdot\btheta+\int_{\O}\nu\bomega\cdot\curl\bv-\int_{\O}\nu\btheta\cdot\curl\bu +\kappa_1\nu_0\!\!\int_{\O}\curl\bu\cdot\curl\bv\nonumber\\
\label{defa1} &\quad +\kappa_2\!\!\int_{\O}\vdiv\bu\vdiv\bv
-\kappa_1\nu_0\!\!\int_{\O}\bomega\cdot\curl\bv-2\!\int_{\O}\beps(\bu)\nabla\nu\cdot\bv+\int_{\O}\bomega\cdot(\nabla\nu\times\bv),\\
\nonumber B((\bv,\btheta),q)& :=-\int_{\O}q\vdiv\bv, \qquad 
G(\bv,\btheta) :=\int_{\O}\ff\cdot\bv,
\end{align}
for all $(\bu,\bomega),(\bv,\btheta)\in\HCUO^3\times\LO^3$, and
$q\in\LOO$. 

\subsection{Unique solvability of the augmented formulation}
Problem \eqref{probform2} accommodates an analysis directly
under the classical Babu\v ska-Brezzi theory~\cite{bbf-2013,G2014}.
More precisely, the continuity of the bilinear and linear functionals in \eqref{defa1}
is a direct consequence of Lemma~\ref{bounds} below, whose 
proof is obtained by rather standard arguments. 
In particular, the penultimate estimate holds owing to 
the assumption $\nabla\nu\in\L^{\infty}(\Omega)^3$ and
the fact that $\Vert\nabla\nu\times\bv\Vert_{0,\O}\le2\Vert\nabla\nu\Vert_{\infty,\O}\Vert\bv\Vert_{0,\O}$. 
Then, the ellipticity of $A$, stated in Lemma~\ref{lem-elip}, follows from adding the redundant terms in \eqref{least0}-\eqref{least1}. 
\begin{lemma}\label{bounds}
The following estimates hold 
\begin{gather*}
\vert \int_{\O}\nu\K^{-1}\bu\cdot\bv\vert \le \sigma_{\max}\nu_1\nnorm{\bu}_{1,\O}\nnorm{\bv}_{1,\O},\qquad 
\vert\int_{\O}\nu\bomega\cdot\btheta\vert \le \nu_1\Vert\bomega\Vert_{0,\O}\Vert\btheta\Vert_{0,\O},\\
\vert\int_{\O}\nu\btheta\cdot\curl\bv\vert \le \nu_1\Vert\btheta\Vert_{0,\O}\nnorm{\bv}_{1,\O},\qquad 
\vert\int_{\O}\beps(\bu)\nabla\nu\cdot\bv\vert \le \Vert\nabla\nu\Vert_{\infty,\O}\Vert\beps(\bu)\Vert_{0,\O}\Vert\bv\Vert_{0,\O},\\
\vert\int_{\O}\btheta\cdot(\nabla\nu\times\bv)\vert \le 2\Vert\nabla\nu\Vert_{\infty,\O}\Vert\bv\Vert_{0,\O}\Vert\btheta\Vert_{0,\O},\qquad 
\vert G(\bv,\btheta)\vert \le \Vert\ff \Vert_{0,\O}\Vert\bv\Vert_{0,\O}.
\end{gather*}
\end{lemma}
\noindent Therefore, we have that there exist $C_1, C_2, C_3 > 0$ such that
\begin{gather*}
\vert A((\bu,\bomega),(\bv,\btheta))\vert\le 
C_{1}\Vert(\bu,\bomega)\Vert_{\HCUO^3\times\LO^3}\Vert(\bv,\btheta)\Vert_{
\HCUO^3\times\LO^3},\\
\vert B((\bv,\btheta),q)\vert\le  C_{2}
\Vert(\bv,\btheta)\Vert_{\HCUO^3\times\LO^3}\Vert q\Vert_{0,\O},\qquad 
\vert G(\bv,\btheta)\vert \le C_{3}\Vert(\bv,\btheta)\Vert_{\HCUO^3\times\LO^3},
\end{gather*}
where
$$\Vert(\bv,\btheta)\Vert_{\HCUO^3\times\LO^3}^2=\nnorm{\bv}_{1,\O}^2+\Vert\btheta\Vert_{0,\O}^2.$$
\begin{lemma}\label{lem-elip}
Assume that
\begin{equation}\label{hipo}
\frac{4\Vert\nabla\nu\Vert_{\infty,\O}^2}{\sigma_{\min}\nu_0^2}< 1/4.
\end{equation}
Suppose that $\kappa_1\in(\frac{1}{2},\frac{3}{2})$ and $\kappa_2>\frac{\nu_0}{4}$. 
Then, there exists $\alpha>0$ such that
$$A((\bv,\btheta),(\bv,\btheta))\ge
\alpha\Vert(\bv,\btheta)\Vert_{\HCUO^3\times\LO^3}^2\qquad\forall(\bv,
\btheta)\in\HCUO^3\times\LO^3.$$
\end{lemma}

\begin{proof}
Given $(\bv,\btheta)\in\HCUO^3\times\LO^3$ first we observe that
as a consequence of Lemma~\ref{bounds}, we have 
\begin{equation*}
\left\vert2\int_{\O}\beps(\bv)\nabla\nu\cdot\bv\right\vert
\le \frac{4\Vert\nabla\nu\Vert_{\infty,\Omega}^2}{\sigma_{\min}\nu_0}(\Vert\curl\bv\Vert_{0,\Omega}^2
+\Vert\vdiv\bv\Vert_{0,\O}^2)+\frac{\sigma_{\min}\nu_0}{4}\Vert\bv\Vert_{0,\Omega}^2,
\end{equation*}
where we have used \eqref{poinc}. Moreover,
using that $\Vert(\nabla\nu\times\bv)\Vert_{0,\Omega}\le
2\Vert\nabla\nu\Vert_{\infty,\Omega}\Vert\bv\Vert_{0,\Omega}$, we obtain
\begin{align*}
\left\vert\int_{\O}\btheta\cdot(\nabla\nu\times\bv)\right\vert
\le& \frac{4\Vert\nabla\nu\Vert_{\infty,\Omega}^2}{\sigma_{\min}\nu_0}\Vert\btheta\Vert_{0,\Omega}^2
+\frac{\sigma_{\min}\nu_0}{4}\Vert\bv\Vert_{0,\Omega}^2,\\
\left\vert\kappa_1\nu_0\int_{\O}\btheta\cdot\curl\bv\right\vert
\le& \frac{\kappa_1\nu_0}{2}\Vert\btheta\Vert_{0,\Omega}^2+\frac{\kappa_1\nu_0}{2}\Vert\curl\bv\Vert_{0,\Omega}^2,
\end{align*}
and these estimates are put in combination with Cauchy-Schwarz inequality
to obtain that
\begin{align*}
A((\bv,\btheta),(\bv,\btheta))\geq & \, \sigma_{\min}\nu_0\Vert\bv\Vert_{0,\O}^2
+\nu_0\Vert\btheta\Vert_{0,\O}^2+\kappa_1\nu_0\Vert\curl\bv\Vert_{0,\O}^2
-\frac{\kappa_1\nu_0}{2}\Vert\curl\bv\Vert_{0,\O}^2-\frac{\kappa_1\nu_0}{2}\Vert\btheta\Vert_{0,\O}^2\\
&+\kappa_2\Vert\vdiv\bv\Vert_{0,\O}^2-\frac{4\Vert\nabla\nu\Vert_{\infty,\Omega}^2}{\sigma_{\min}\nu_0}(\Vert\curl\bv\Vert_{0,\Omega}^2
+\Vert\vdiv\bv\Vert_{0,\O}^2)-\frac{\sigma_{\min}\nu_0}{4}\Vert\bv\Vert_{0,\Omega}^2\\
&-\frac{4\Vert\nabla\nu\Vert_{\infty,\Omega}^2}{\sigma_{\min}\nu_0}\Vert\btheta\Vert_{0,\O}^2-\frac{\sigma_{\min}\nu_0}{4}\Vert\bv\Vert_{0,\O}^2\\
\ge&\frac{\sigma_{\min}\nu_0}{2}\Vert\bv\Vert_{0,\O}^2+\biggl((1-\frac{\kappa_1}{2})\nu_0-\frac{4\Vert\nabla\nu\Vert_{\infty,\Omega}^2}{\sigma_{\min}\nu_0}\biggr)\Vert\btheta\Vert_{0,\O}^2
+\biggl(\frac{\kappa_1\nu_0}{2}-\frac{4\Vert\nabla\nu\Vert_{\infty,\Omega}^2}{\sigma_{\min}\nu_0}\biggr)\Vert\curl\bv\Vert_{0,\O}^2\\
&+\left(\kappa_2-\frac{4\Vert\nabla\nu\Vert_{\infty,\Omega}^2}{\sigma_{\min}\nu_0}\right)\Vert\vdiv\bv\Vert_{0,\O}^2.
\end{align*}
Now, using \eqref{hipo}, we have that
\begin{align*}
A((\bv,\btheta),(\bv,\btheta))
\ge&\frac{\sigma_{\min}\nu_0}{2}\Vert\bv\Vert_{0,\O}^2+\frac{\nu_0}{2}\left(\frac{3}{2}-\kappa_1\right)\Vert\btheta\Vert_{0,\O}^2
+\frac{\nu_0}{2}\left(\kappa_1-\frac{1}{2}\right)\Vert\curl\bv\Vert_{0,\O}^2+\left(\kappa_2-\frac{\nu_0}{4}\right)\Vert\vdiv\bv\Vert_{0,\O}^2\\
\ge&\min\left\{\frac{\sigma_{\min}\nu_0}{2},\frac{\nu_0}{2}\left(\kappa_1-\frac{1}{2}\right),\left(\kappa_2-\frac{\nu_0}{4}\right)\right\}\nnorm{\bv}_{1,\O}^2
+\frac{\nu_0}{2}\left(\frac{3}{2}-\kappa_1\right)\Vert\btheta\Vert_{0,\O}^2\\
\ge&\alpha\Vert(\bv,\btheta)\Vert_{\HCUO^3\times\LO^3}^2,
\end{align*}
where $\alpha$ depends on $\kappa_1,\kappa_2,\nu_0$ and $\sigma_{\min}$.
\end{proof}

Finally, recall the inf-sup condition (cf. \cite{G2014}): there exists $C>0$ only depends on $\O$ such that
\begin{equation}\label{inf-sup-classical}
\sup_{0\ne\bv\in\HCUO^3}\frac{\vert\int_{\O}
q\vdiv\bv\vert}{\Vert\bv\Vert_{1,\O}}\ge C\Vert
  q\Vert_{0,\O}\quad\forall q\in\LOO.
\end{equation}
\begin{lemma}\label{inf-sup-cont}
There exists $\beta>0$, independent of $\nu$, such that
\begin{equation*}
\sup_{0\ne(\bv,\btheta)\in\HCUO^3\times\LO^3}\frac{\vert
  B((\bv,\btheta),q)\vert}{\Vert(\bv,\btheta)\Vert_{\HCUO^3\times\LO^3}}\ge \beta\Vert
  q\Vert_{0,\O}\quad\forall q\in\LOO.
\end{equation*}
\end{lemma}
\begin{proof}
The result is a consequence of \eqref{inf-sup-classical}
and the fact that 
$$\nnorm{\bv}_{1,\O}\le\Vert\bv\Vert_{1,\O},$$
where the term in the righ-hand side has the usual norm in $\HCUO^3$.
\end{proof}

\noindent All these steps lead to the unique solvability of the problem. 
\begin{theorem}
There exists a unique solution
$((\bu,\bomega),p)\in(\HCUO^3\times\LO^3)\times\LOO$
to \eqref{probform2} and there exists
a constant $C>0$ such that
the following continuous dependence result holds:
$$ \Vert((\bu,\bomega),p)\Vert_{(\HCUO^3\times\LO^3)\times\LO}
\le C\Vert\ff\Vert_{0,\O}.$$
\end{theorem}

\begin{proof}
By virtue of Lemmas~\ref{lem-elip} and \ref{inf-sup-cont},
the proof is a straightforward application of
\cite[Thm.~II.1.1]{bbf-2013}.
\end{proof}

\section{Finite element discretisations}\label{sec:FEM}
Taking generic subspaces for the approximation of velocity, vorticity, and pressure, a 
Galerkin scheme associated with \eqref{probform2} reads: 
{\em Find $((\bu_h,\bomega_h),p_h)\in(\H_h\times\Z_h)\times\Q_h$ such that}
\begin{equation}\label{probdisc}
\begin{split}
A((\bu_h,\bomega_h),(\bv_h,\btheta_h))+B((\bv_h,\btheta_h),p_h)=&\;G(\bv_h,
\btheta_h)\quad\forall(\bv_h,\btheta_h)\in\H_h\times\Z_h,\\
B((\bu_h,\bomega_h),q_h)=&\;0\qquad\forall q_h\in\Q_h.
\end{split}
\end{equation}
We can adopt in particular 
\begin{align}
\H_h&:=\{\bv_h \in \HO^3:
\bv_h|_{T} \in \bbP_{k+1}(T)^3, \ \forall T \in \CT_{h} \}\cap\HCUO^3,\label{esp1}\\
\Z_h&:=\left\{\btheta_{h}\in \LO^3: \btheta_h|_{T}\in\bbP_{\ell}(T)^3, \ \forall T \in \CT_{h}  \right\},\label{esp2}\\
\Q_h&:=\{ q_h\in \HO: q_h|_{T} \in \bbP_{k}(T) , \ \forall T \in \CT_{h} \}\cap\LOO\label{esp3},
\end{align}
where $k\geq 1, \ell \geq 0$. Here 
$\{\cT_{h}(\O)\}_{h>0}$ is a shape-regular
family of partitions of 
$\bar\O$ by tetrahedra $T$ of diameter $h_T$. The meshsize is 
$h:=\max\{h_T:\; T\in\cT_{h}(\O)\}$, and $\bbP_m(S)$ denotes the space of polynomials 
with total degree up to $m$, defined on a generic set $S$. 

We recall that $\H_h\times\Q_h$ in the generalised Hood-Taylor finite element
pair for the Stokes equations \cite{HT}. As we will see, the schemes 
coming from \eqref{probdisc}-\eqref{esp3} are well-posed for any approximation 
order of the discrete vorticity $\ell$ (and being continuous or discontinuous polynomials); 
however, an appropriate choice is to take $\ell = k$ and discontinuous elements, which 
deliver a consistent overall rate of convergence for all unknowns. 

Next, we proceed to show that the proposed method is stable and convergent. 
\begin{lemma}\label{coerdisc}
Assuming \eqref{hipo}, and choosing 
$\kappa_1\in\left(\frac{1}{2},\frac{3}{2}\right)$ and $\kappa_2>\frac{\nu_0}{4}$, 
 there exists $\alpha>0$, such that
$$A((\bv_h,\btheta_h),(\bv_h,\btheta_h))\ge
\alpha\Vert(\bv_h,\btheta_h)\Vert_{\HCUO^3\times\LO^3}^2\qquad\forall(\bv_h,
\btheta_h)\in\H_h\times\Z_h.$$
\end{lemma}

\begin{remark}
The values for the augmentation parameters $\kappa_1$ and $\kappa_2$ 
are chosen such that the largest ellipticity constant
in Lemma~\ref{coerdisc} is achieved. This means that we take $\kappa_1=1$
(the middle point of the relevant interval, see e.g. \cite[Sect. 3]{anaya15b}) 
and $\kappa_2=\frac{\nu_0}{2}$.
\end{remark}

\noindent Moreover, since for the pair of spaces \eqref{esp1},\eqref{esp3} one has an inf-sup condition of the form 
\begin{equation}\label{eq:inf-sup-b2-h}
\sup_{\stackrel{\scriptstyle\bv_h\in\H_h}{\bv_h\ne0}}\frac{\int_{\O}q_h\vdiv\bv_h}{\Vert\bv_h\Vert_{1,\O}}\ge
\tilde\beta_2\Vert q_h\Vert_{0,\O}\quad\forall q_h\in\Q_h,
\end{equation}
where $\tilde\beta_2$ is independent
of $h$ (see \cite{Boffi94,Boffi97}), then it is straightforward to prove the following result. 
\begin{lemma}\label{infsupdisc}
There exists $\tilde{\beta}>0$, such that
$$\sup_{\stackrel{\scriptstyle(\bv_h,\btheta_h)\in\H_h\times\Z_h}{(\bv_h,
\btheta_h)\ne0}}
\frac{\vert B((\bv_h,\btheta_h),q_h)\vert}{\Vert(\bv_h,\btheta_h)\Vert_{\HCUO^3\times\LO^3}
}\ge
\tilde{\beta}\Vert q_h\Vert_{0,\O}\quad\forall q_h\in\Q_h.$$ 
\end{lemma}

\noindent Recall now that the Lagrange interpolant $\Pi:\HusO^3\to\H_h$ 
satisfies the following error estimate:
There exists $C>0$, independent of $h$, such that
for all $s\in(1/2,k+1]$:
\begin{equation}\label{prop2lagrange}
\|\bv-\Pi\bv\|_{1,\O}
\le Ch^{s}\|\bv\|_{1+s,\O}\quad\forall\bv\in\HusO^3.
\end{equation}
Likewise, denoting by $\CP$ the orthogonal projection from $\LO$ (or from $\LO^3$)
onto the subspace $\Q_h$ (or onto the subspace $\Z_h$), we have 
an estimate valid for all $s>0$:
\begin{equation}\label{cotar}
\Vert q-\CP q\Vert_{0,\O}\le Ch^{s}\Vert q\Vert_{s,\O}\quad\forall q\in\HsO.
\end{equation}

Thanks to Lemmas \ref{coerdisc} and \ref{infsupdisc}, we 
can state the stability and C\'ea estimate of the method as follows. 
\begin{theorem}
Let  $\H_h$, $\Z_h$ and $\Q_h$ be specified as 
in \eqref{esp1}, \eqref{esp2} and \eqref{esp3}, respectively.
Then, there exists a unique $((\bu_h,\bomega_h),p_h)\in(\H_h\times\Z_h)\times\Q_h$
solution of the Galerkin scheme~\eqref{probdisc}.
Furthermore, there exist positive constants $\hat{C}_1,\,\hat{C}_2>0$,
independent of $h$, such that
\begin{equation}\label{eq:stability-h}
\Vert(\bu_h,\bomega_h)\Vert_{\HCUO^3\times\LO^3}+\Vert p_h\Vert_{0,\O}\le
\hat{C}_1\Vert\ff\Vert_{0,\O},
\end{equation}
and
\begin{align}\label{ceaest}
\begin{split}
&\Vert(\bu,\bomega)-(\bu_h,\bomega_h)\Vert_{\HCUO^3\times\LO^3}
+\Vert p-p_h\Vert_{0,\O}\\
&\qquad\qquad\qquad\le\hat{C}_2\inf_{(\bv_h,\btheta_h,q_h)\in\H_h\times\Z_h\times\Q_h}
(\nnorm{\bu-\bv_h}_{1,\O}+\Vert \bomega-\btheta_h\Vert_{0,\O}+\Vert p-q_h\Vert_{0,\O}),
\end{split}
\end{align}
where $((\bu,\bomega),p)\in(\HCUO^3\times\LO^3)\times\LOO$ is the
unique solution to variational problem \eqref{probform2}.
\end{theorem}

And finally the convergence of the augmented scheme can be formulated as follows. 
\begin{theorem}
Let $\H_h,\Z_h$ and $\Q_h$
be given by \eqref{esp1}, \eqref{esp2}, and \eqref{esp3}, respectively, setting $\ell=k$
with $k\ge1$.
Let $(\bu,\bomega,p)\in\HCUO^3\times\LO^3\times\LOO$ and
$(\bu_h,\bomega_h,p_h)\in\H_h\times\Z_h\times\Q_h$ be the unique solutions
to the continuous and discrete problems \eqref{probform2} and
\eqref{probdisc}, respectively.  Assume that
$\bu\in\HusO^3$, $\bomega\in\HsO^3$ and $p\in\HsO$, for some
$s\in(1/2,k+1]$. Then, there exists $\hat{C}>0$, independent of $h$, such
  that
\begin{equation}\label{convergence}
\Vert(\bu,\bomega)-(\bu_h,\bomega_h)\Vert_{\HCUO^3\times\LO^3}
+\Vert p-p_h\Vert_{0,\O}\leq\hat{C}h^{s}(\Vert\bu\Vert_{\HusO^3}
+\|\bomega\|_{\HsO^3}+\Vert p\Vert_{\HsO}).\end{equation}
\end{theorem}

\begin{proof}
It follows from \eqref{eq:stability-h}-\eqref{ceaest} and \eqref{prop2lagrange}-\eqref{cotar}.
\end{proof}

\begin{remark}\label{mini}
Instead of Hood--Taylor finite elements \eqref{esp1},\eqref{esp3}, 
we  can also consider any other Stokes inf-sup stable pairs. For instance, using  
the MINI-element for velocity and pressure (piecewise linear velocities enriched 
with quartic bubbles, or cubic bubbles in 2D, and piecewise linear and continuous pressures, see e.g. \cite{bbf-2013}) and piecewise constant 
elements for vorticity, we can easily adapt the analysis  to obtain the 
error estimate 
$$\Vert(\bu,\bomega)-(\bu_h,\bomega_h)\Vert_{\HCUO^3\times\LO^3}
+\Vert p-p_h\Vert_{0,\O}\leq\hat{C}h^s(\Vert\bu\Vert_{\HusO^3}
+\|\bomega\|_{\HsO}+\Vert p\Vert_{\HsO}).$$
\end{remark}

\section{Numerical results}\label{sec:results}
\begin{figure}[h]
\begin{center}
\includegraphics[width=0.325\textwidth]{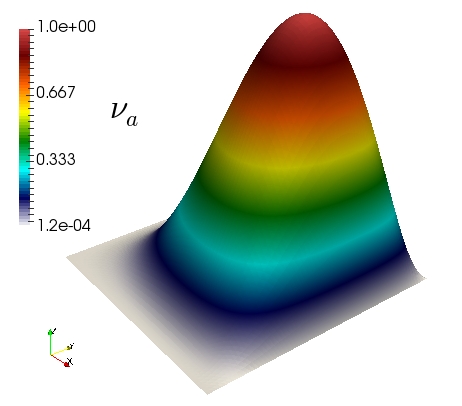}
\includegraphics[width=0.325\textwidth]{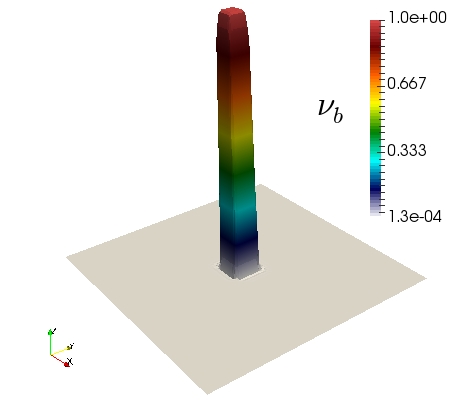}\\
\includegraphics[width=0.325\textwidth]{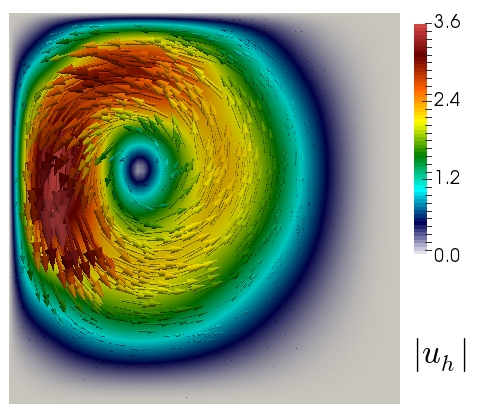}
\includegraphics[width=0.325\textwidth]{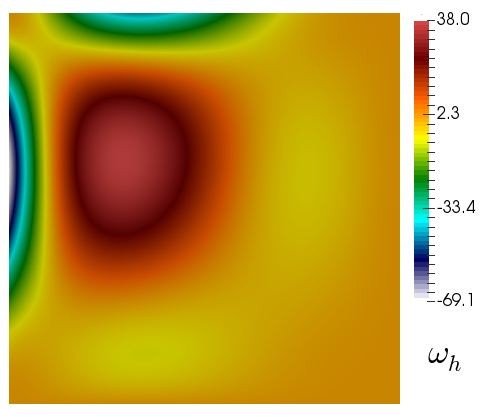}
\includegraphics[width=0.325\textwidth]{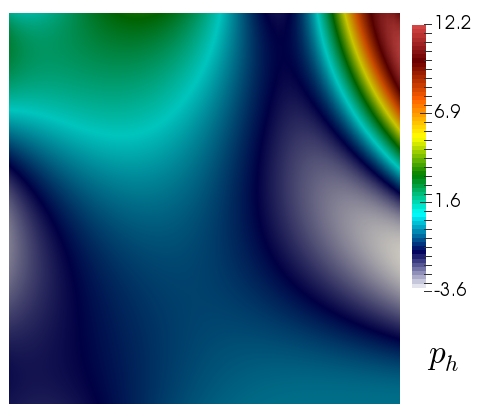}
\end{center}
\vspace{-3mm}
\caption{Smooth and steep viscosity profiles $\nu_a,\nu_b$ (top), 
and approximate solutions generated, for $\nu_b$, employing a second-order scheme.}\label{fig:ex01}
\end{figure}
 
We proceed to verify numerically the 
 convergence rates predicted by \eqref{convergence}.  Following \cite{john15}, 
 on $\Omega = (0,1)^2$   
 we take $\phi(x,y) = 1000x^2(1-x)^4y^3(1-y)^2$ and define exact velocity, vorticity, and pressure as 
 $$\bu = (\partial_y\phi,-\partial_x\phi)^{\tt t}, \quad \bomega  = \curl \bu,\quad p(x,y) = \pi^2 (xy^2\cos(2\pi x^2y) - x^2y\sin(2\pi xy)) - \frac{1}{8},$$
 which satisfy the incompressibility constraint as well as the homogeneous boundary and compatibility conditions. Two specifications 
 for viscosity are considered, with a mild and with a higher gradient
 $$\nu_a(x,y) = \nu_0 + (\nu_1 - \nu_0) x^2 (1-x)y^2(1-y)\frac{721}{16}, \quad \nu_b(x,y) = \nu_0 + (\nu_1 - \nu_0) \exp(-10^{13}[(x-0.5)^{10}+(y-0.5)^{10}]),$$
 and we use $\nu_0 = 10^{-4}$, $\nu_1=1$. A current restriction in our analysis is \eqref{hipo} that only permits sufficiently small permeability such that the lower bound for its inverse, $\sigma_{\min}$ is  large enough (in any case, for most relevant applications in porous media flow these values are reasonable). 
 We use $\K = 10^{-6}\bI$.  
 Sample solutions are shown in Figure~\ref{fig:ex01} and 
the convergence history (produced on a sequence of successively refined meshes and computing errors for all fields and rates as usual) 
is presented in Table~\ref{table:ex01}. At least for these two cases, we observe a higher convergence of the pressure and 
that a steeper viscosity does not affect the accuracy. 

\begin{table}[h]
\begin{center}
{\small \begin{tabular}{|l|c|c|c|c|c|c|c|}
\hline
DoF  & $h$ &   $\|\bu-\bu_h\|_{1,\O}$  &  \texttt{rate}  &    $\|\bomega-\bomega_h\|_{0,\Omega}$  &  \texttt{rate} &  $\|p-p_h\|_{0,\O}$  &  \texttt{rate}  \\  
\hline
\multicolumn{8}{|c|}{smooth viscosity $\nu_a$}\\
\hline
84 & 0.7071 & 11.233 &   --  & 10.580 &   -- & 2126 &    -- \\
284 & 0.3536 & 4.4150 & 1.347 & 3.6531 & 1.524 & 1194 & 0.832 \\
1044 & 0.1768 & 1.2351 & 1.838 & 1.0024 & 1.863 & 271.24 & 2.136 \\
4004 & 0.0884 & 0.3092 & 1.999 & 0.2482 & 2.016 & 44.490 & 2.609 \\
15684 & 0.0442 & 0.0767 & 2.011 & 0.0609 & 2.027 & 6.2553 & 2.732 \\
62084 & 0.0221 & 0.0191 & 2.005 & 0.0150 & 2.015 & 0.8594 & 2.525 \\
247044 & 0.0111 & 0.0047 & 1.999 & 0.0037 & 2.008 & 0.2503 & 2.318 \\
\hline
\multicolumn{8}{|c|}{steeper viscosity $\nu_b$}\\
\hline
84 & 0.7071 & 11.233 &    -- & 10.581 &   -- & 2125 &    -- \\
284 & 0.3536 & 4.4150 & 1.347 & 3.6528 & 1.524 & 1193 & 0.832\\
1044 & 0.1768 & 1.2350 & 1.837 & 1.0024 & 1.862 & 271.25 & 2.136\\
4004 & 0.0884 & 0.3093 & 1.998 & 0.2484 & 2.016 & 44.491 & 2.609\\
15684 & 0.0442 & 0.0767 & 2.011 & 0.0609 & 2.027 & 6.2553 & 2.731\\
62084 & 0.0221 & 0.0191 & 2.005 & 0.0151 & 2.016 & 0.8603 & 2.437\\
247044 & 0.0111& 0.0048 & 1.999 & 0.0037 & 2.008 & 0.2487 & 2.290\\
\hline
\end{tabular}}\end{center}
\caption{Error history associated to the augmented scheme using \eqref{esp1}-\eqref{esp3} with 
$k = \ell = 1$.} \label{table:ex01}
\end{table}

We close with a 3D example simulating the cavity flow in the presence of a viscosity boundary 
layer. The domain $\Omega = (0,1)^3$ is discretised with a structured tetrahedral mesh 
and we employ the scheme from Remark~\ref{mini} (the 
MINI-element for the velocity-pressure pair together with piecewise 
constant vorticity approximation) resulting in a system with 560165 DoF. We use $\ff=\cero$ and the velocity $\bu = (1,0,0)^{\tt t}$ 
is prescribed on the top lid (at $z=1$) while no-slip velocities are set on the other sides of the boundary. We set 
$\K = 10^{-4}\bI$, and 
choosing now $\nu_0 = 10^{-5}$, $\nu_1=10$, 
the variable viscosity field is  
$$\nu= \nu_0 + (\nu_1 - \nu_0) \exp(-10^{3}[(x-0.1)^6+(y-0.5)^6+(z-0.5)^6]).$$ 
The approximate solutions are depicted in Figure~\ref{fig:ex02} where we observe how the velocity and pressure lose the usual 
symmetry expected in lid-driven cavity flows, and it separates due to the viscosity boundary layer. 

\begin{figure}[h]
\begin{center}
\includegraphics[width=0.245\textwidth]{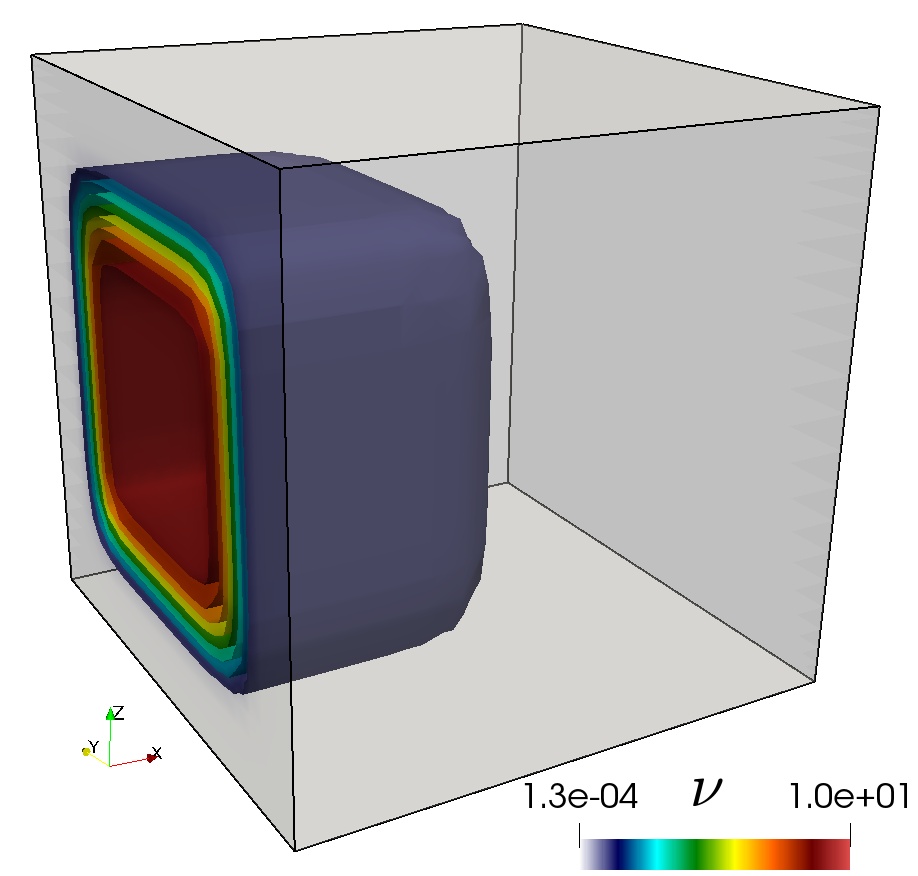}
\includegraphics[width=0.245\textwidth]{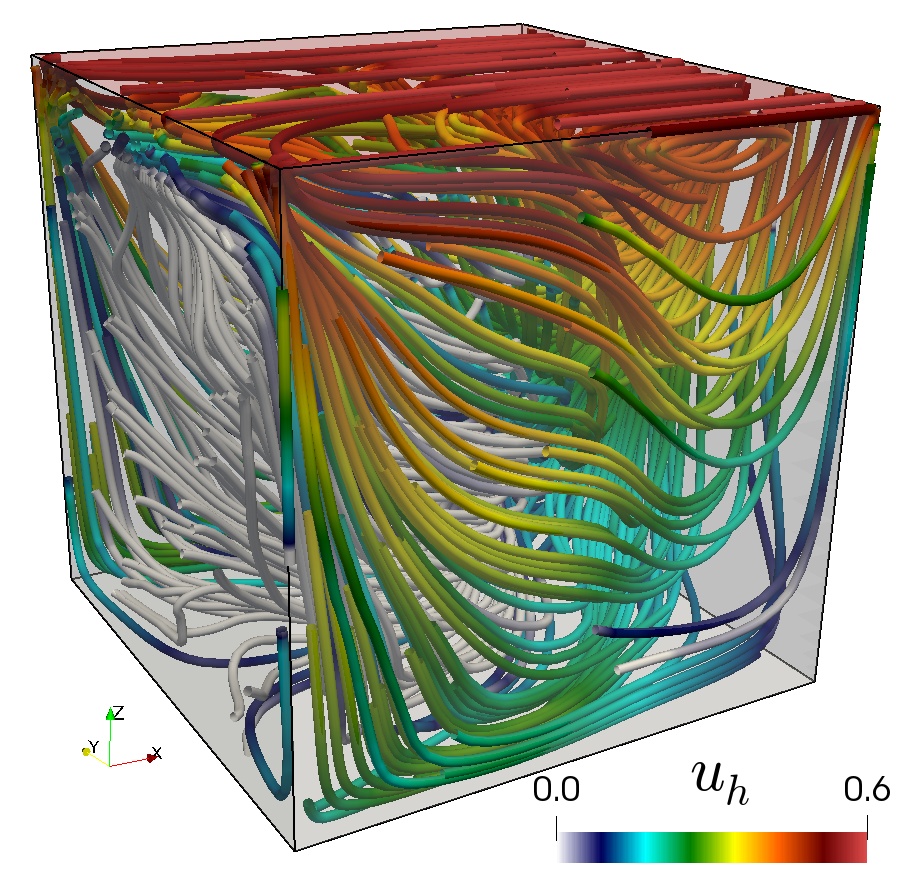}
\includegraphics[width=0.245\textwidth]{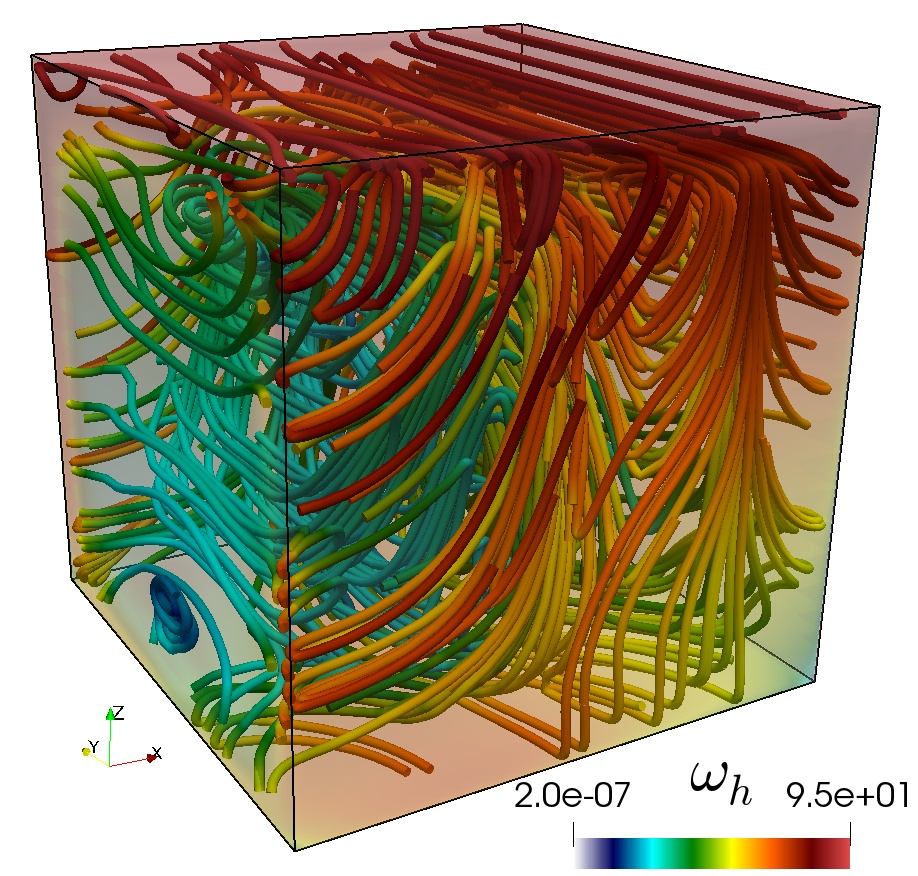}
\includegraphics[width=0.245\textwidth]{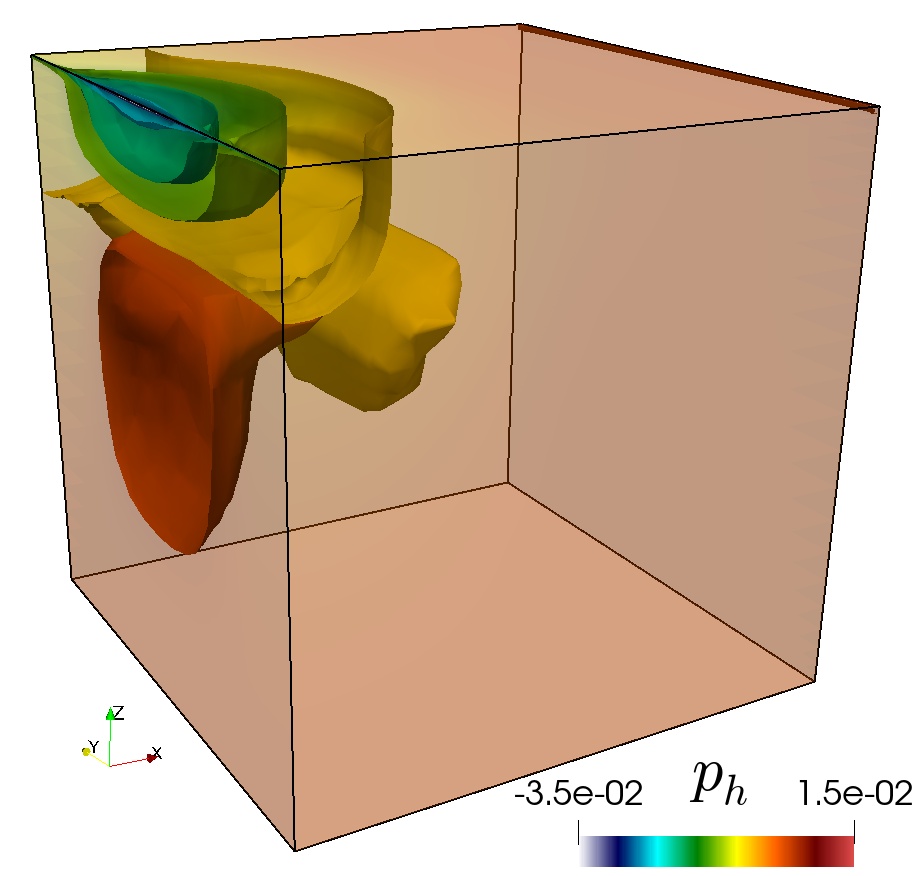}
\end{center}
\vspace{-3mm}
\caption{Viscosity contour, velocity streamlines, vorticity streamlines, and pressure 
computed using the MINI-element.}\label{fig:ex02}
\end{figure}


\end{document}